\documentclass[10pt,a4paper]{article}

\usepackage{epsf,epsfig,amsfonts,amsgen,amsmath,amstext,amsbsy,
	amsopn,amsthm,amssymb, epstopdf
}
\usepackage{ebezier,eepic}
\usepackage{color}
\usepackage{multirow}
\usepackage{url, tikz}
\usepackage{cite}
\usepackage{stmaryrd}
\newcommand \brk[1]
{
	\llbracket #1\rrbracket
}

\usepackage{kantlipsum}
\allowdisplaybreaks

\newtheorem{theorem}{Theorem}

\newtheorem{lemma}{Lemma}
\newtheorem{false statement}{False statement}

\theoremstyle{definition}

\newtheorem{conjecture}{Conjecture}

\newtheorem{problem}{Problem}

\newcommand{\T}{\mathcal{T}}

\newcommand \equ[2]
{
\begin{equation}\label{#1}
#2
\end{equation}
}

\newcommand \aln[2]
{
\begin{align}\label{#1}
#2
\end{align}
}

\newcommand \eqn[2]
{
\begin{eqnarray}\label{#1}
#2
\end{eqnarray}
}

\newcommand \notwant[1] {}

\begin{document}

\title{Counting spanning trees in a complete bipartite graph which contain a given spanning forest}
\date{}

\author{
Fengming Dong$^{1}$\thanks{Email: fengming.dong@nie.edu.sg and donggraph@163.com.} \ and
Jun Ge$^{2}$\thanks{Corresponding author. Email: mathsgejun@163.com.}
\\[2mm]
\small $^{1}$National Institute of Education, Nanyang Technological University, Singapore \\
\small $^{2}$School of Mathematical Sciences \& Laurent Mathematics Center, Sichuan Normal University, China
}

\maketitle

\begin{abstract}
In this article, we extend Moon's classic formula
for counting spanning trees in complete graphs
containing a fixed spanning forest
to complete bipartite graphs.
Let $(X,Y)$ be the bipartition of
the complete bipartite graph $K_{m,n}$
with $|X|=m$ and $|Y|=n$.
We prove that  for any given spanning forest $F$
of $K_{m,n}$ with components $T_1,T_2,\ldots,T_k$,
the number of spanning trees in $K_{m,n}$
which contain all edges in $F$ is equal to
$$
\frac 1{mn}\left(\prod_{i=1}^k (m_in+n_im)\right)
\left (1-\sum_{i=1}^{k}\frac{m_in_i}{m_in+n_im}\right ),
$$ \label{r1-1}
where $m_i=|V(T_i)\cap X|$ and $n_i=|V(T_i)\cap Y|$
for $i=1,2,\ldots,k$.
\end{abstract}

\medskip



\section{Introduction}\label{sec1}

In this paper, we assume that all graphs are loopless, while parallel edges are allowed.
For any graph $G$, let $V(G)$ and $E(G)$ be the vertex set
and edge set of $G$.
For any edge set $F\subseteq E(G)$,
let $G/F$ be the graph obtained from $G$ by
contracting all edges in $F$, and removing all loops.
Let $\T(G)$ denote the set of spanning trees of $G$.
For any positive integer $k$,
let $\brk{k}$ denote the set
$\{1,2,\cdots,k\}$.

Suppose $G$ is a weighted graph with weight function
$\omega: E(G)\rightarrow \mathbb{R}$.
For any $F\subseteq E(G)$ and any subgraph $H$ of $G$,
define $\omega(F)=\prod\limits_{e\in F} \omega(e)$ and $\omega(H)=\omega(E(H))$.
Let $\tau(G,\omega)=\sum\limits_{T\in \mathcal{T}(G)} \omega(T)$.
Sometimes we use $\tau(G)$ instead of $\tau(G,\omega)$ when there is no confusion.
It is obvious that for an unweighted graph $G$
(that is to say, a weighted graph with unit weight on
each edge), $\tau(G)=|\mathcal{T}(G)|$, i.e., the number of spanning trees of $G$.
Throughout this paper, any graph is assumed to be unweighted,
unless it is claimed.

Counting spanning trees in graphs is a very old topic in graph theory
having modern connections with many other fields in mathematics,
statistical physics and theoretical computer science, such as random walks,
the Ising model and Potts model, network reliability, parking functions,
knot/link determinants. See \cite{DGO, DY, Ehrenborg, GJ, Yan} for some recent
work on counting spanning trees.

For a subgraph $H$ of $G$,
let $\mathcal{T}_{H}(G)$ denote
the set of spanning trees $T\in \T(G)$
with $E(H)\subseteq E(T)$,
and let
$\tau_{H}(G)=\sum\limits_{T\in \mathcal{T}_{H}(G)} \omega(T)$.
For unweighted graph $G$, $\tau_{H}(G)=|\mathcal{T}_{H}(G)|$,
i.e., the number of spanning trees of $G$ containing all edges in $H$.
Note that usually the graph $H$ here is a forest or a tree,
because otherwise
$\mathcal{T}_{H}(G)=\emptyset$ and $\tau_{H}(G)=0$.

The celebrated Cayley's formula \cite{Cayley} states that $\tau(K_n)=n^{n-2}$.
In 1964, Moon generalized Cayley's formula by
obtaining a nice expression of $\tau_F(K_n)$ for any spanning forest $F$ of $K_n$. 

\begin{theorem}[\cite{Moon}, also see Problem 4.4 in \cite{Lovasz}]\label{Moon}
For any spanning forest $F$ of $K_n$, if $c$ is the number of components
of $F$ and $n_1, n_2, \ldots, n_c$ are the orders of those components, then
$$\tau_{F}(K_n)=n^{c-2}\prod_{i=1}^c n_i.$$
\end{theorem}

It is easy to see that Cayley's formula is the special case
that $F$ is an empty graph.
It is also well known that $\tau(K_{m, n})=m^{n-1}n^{m-1}$
for any complete bipartite graph $K_{m, n}$
by Fiedler and Sedl\'{a}\v{c}ek \cite{FS}. So there is
a natural question: is there a bipartite analogue of
Moon's formula (Theorem \ref{Moon})?
That is to say,
for any given spanning forest $F$ in $K_{m,n}$,
what is the explicit expression of $\tau_F(K_{m, n})$?

It turns out that this question is much harder than the case of complete graphs.
In \cite{GD}, this question was partially answered for two special cases:
$F$ is a matching or a tree plus several possible isolated vertices.

\begin{theorem}[\cite{GD}]\label{matching}
For any matching $M$ of size $k$ in $K_{m, n}$,
$$\tau_M(K_{m, n})=(m+n)^{k-1}(m+n-k)m^{n-k-1}n^{m-k-1}.$$
\end{theorem}

\begin{theorem}[\cite{GD}]\label{tree}
For any tree $T$ of $K_{m, n}$,
$$
\tau_T(K_{m, n})=(sn+tm-st)m^{n-t-1}n^{m-s-1},
$$
where $s=|V(T)\cap X|$, $t=|V(T)\cap Y|$,
and $(X, Y)$ is the bipartition
of $K_{m, n}$ with $|X|=m$ and $|Y|=n$.
\end{theorem}

In this paper, we obtain an explicit expression
for $\tau_F(K_{m,n})$ for an arbitrary spanning forest $F$
of $K_{m,n}$.

\begin{theorem} \label{mainthm}
Let $(X, Y)$ be the bipartition
of $K_{m, n}$ with $|X|=m$ and $|Y|=n$.
For any spanning forest $F$ of $K_{m,n}$
with components $T_1,T_2,\ldots,T_k$,
\equ{main}
{
\tau_{F}(K_{m,n})=
\frac 1{mn}\left(\prod_{i=1}^k (m_in+n_im)\right)
\left (1-\sum_{i=1}^{k}\frac{m_in_i}{m_in+n_im}\right ),
}
where $m_i=|X\cap V(T_i)|$ and $n_i=|Y\cap V(T_i)|$
for all $i\in \brk{k}$.
\end{theorem}

\section{Preliminary results}\label{sec2}

This section provides some results which will be
applied in the next section for proving a key identity.

\begin{lemma}\label{le2-01}
For any set of $k$ pairs of real numbers
$\{a_i,b_i: i\in \brk{k} \}$, where $k\ge 1$,
if $a_iB+b_iA\ne 0$ for all
$i\in \brk{k}$
where $A=a_1+a_2+\cdots+a_k$ and $B=b_1+b_2+\cdots+b_k$, then  \label{r2-3}
\equ{le2-01e1}
{
\left (
1-\sum_{i=1}^k\frac{a_ib_i}{a_iB+b_iA}
\right )^2
=\left(\sum_{i=1}^k\frac{a_i^2}{a_iB+b_iA}\right)
\cdot \left(\sum_{i=1}^k\frac{b_i^2}{a_iB+b_iA}\right).
}
\end{lemma}

\begin{proof}
Since $a_iB+b_iA\ne 0$, $A$ and $B$ cannot both be 0.

When $A=0$ and $B\neq 0$, equality (\ref{le2-01e1}) holds because
$$\left(1-\sum_{i=1}^k\frac{a_ib_i}{a_iB+b_iA}\right)^2=\left(1-\sum_{i=1}^k\frac{b_i}{B}\right)^2=0,$$
and
$$\left(\sum_{i=1}^k\frac{a_i^2}{a_iB+b_iA}\right)\cdot
\left(\sum_{i=1}^k\frac{b_i^2}{a_iB+b_iA}\right)=\frac{A}{B}\cdot\sum_{i=1}^k\frac{b_i^2}{a_iB}=0.$$

Similarly, equality (\ref{le2-01e1}) holds when $B=0$ and $A\neq 0$.
Let $W_i=a_iB+b_iA$ for $i\in \brk{k}$.
When $A\neq 0$ and $B\neq 0$, observe that
\aln{eq2-02}
{
A\sum_{i=1}^k\frac{a_ib_i}{W_i}
+B\sum_{i=1}^k\frac{a_i^2}{W_i}
=\sum_{i=1}^k\frac{a_i(Ab_i+Ba_i)}{W_i}
=\sum_{i=1}^ka_i=A,
}
implying that
\equ
{eq2-03}
{
\sum_{i=1}^k\frac{a_i^2}{W_i}
=\frac{A}{B}
\left (1-\sum_{i=1}^k\frac{a_ib_i}{W_i}
\right ).
}
Similarly,
\equ
{eq2-04}
{
\sum_{i=1}^k\frac{b_i^2}{W_i}
=\frac{B}{A}
\left (1-\sum_{i=1}^k\frac{a_ib_i}{W_i}
\right ).
}
Clearly, equality (\ref{le2-01e1}) follows from
(\ref{eq2-03}) and (\ref{eq2-04}).
\end{proof}

For a set $A$ of real numbers,
in the following, if $A$ is the empty set, we set
\equ{setting}
{
\prod_{a\in A}a=1 \quad \mbox{and}\quad
\sum_{a\in A}a=0.
}

\begin{lemma}\label{le2-0}
Let $S$ be a set of positive integers.
For any set of $2|S|$ real numbers $\{a_i,b_i: i\in S\}$,
\begin{equation} \label{le2-0-1}
\sum_{\emptyset \ne I\subseteq S}
\left (
\left(\prod_{j\in I}a_j\right)
\left(\prod_{r\in S\setminus I}b_r\right )
\right )
=\prod_{i\in S}(a_i+b_i)
-
\prod_{i\in S}b_i.
\end{equation}
\end{lemma}

\begin{proof}
The identity follows from the following fact:
\begin{equation} \label{le2-0-2}
\sum_{\emptyset\subseteq I\subseteq S}
\left (
\left(\prod_{j\in I}a_j\right )
\left(\prod_{r\in S\setminus I}b_r\right )
\right )
=\prod_{i\in S}(a_i+b_i).
\end{equation}
\end{proof}

\begin{lemma}\label{le2-1}
Let $S$ be a set of positive integers.
For any set of $3|S|$
real numbers $\{a_i,b_i,c_i: i\in S\}$,
\begin{equation} \label{le2-1-1}
\sum_{\emptyset \ne I\subseteq S}
\left (
\left (\sum_{i\in I}c_i \right )
\left(\prod_{j\in I}a_j\right )
\left (\prod_{r\in S\setminus I}b_r\right )
\right )
=\left(\prod_{j\in S}(a_j+b_j)\right)
\cdot
\sum_{i\in S}
\frac{c_ia_i}{a_i+b_i}.
\end{equation}
\end{lemma}

\begin{proof}
The identity follows from the following fact:
\begin{equation} \label{le2-1-2}
\sum_{\emptyset \ne I\subseteq S}
\left (
\left (\sum_{i\in I}c_i \right )
\left( \prod_{j\in I}a_j\right )
\left(\prod_{r\in S\setminus I}b_r\right )
\right )
=\sum_{i\in S} \left (
c_ia_i\prod_{j\in S\setminus \{i\}}(a_j+b_j)
\right ).
\end{equation}
\end{proof}

By Lemmas~\ref{le2-0} and~\ref{le2-1}, for an arbitrary real number $c$, we have
\begin{equation} \label{le2-1-1ex}
\sum_{\emptyset \ne I\subseteq S}
\left (
\left (
c+\sum_{i\in I}c_i \right )
\left(\prod_{j\in I}a_j\right)
\left(\prod_{r\in S\setminus I}b_r\right )
\right )
=\left (\prod_{j\in S}(a_j+b_j)\right )
\cdot
\left (c+
\sum_{i\in S}
\frac{c_ia_i}{a_i+b_i}
\right )
-c\prod_{j\in S}b_j.
\end{equation}

\label{r2-5}
\begin{lemma}\label{le2-11}
Let $S$ be a set of positive integers.
For any set of $3|S|$
real numbers $\{a_i,b_i,d_i: i\in S\}$,
\begin{equation} \label{le2-11-1}
\sum_{\emptyset \ne I\subseteq S}
\left (
\left (\sum_{i\in S\setminus I}d_i \right )
\left (\prod_{j\in I}a_j\right )
\left(\prod_{r\in S\setminus I}b_r\right )
\right )
=\left(\prod_{j\in S}(a_j+b_j)\right)
\cdot
\left(\sum_{i\in S}
\frac{d_ib_i}{a_i+b_i}\right )
-\left( \prod_{r\in S}b_r\right )
\cdot
\left( \sum_{i\in S}d_i\right ).
\end{equation}
\end{lemma}

\begin{proof}
It follows from Lemma~\ref{le2-1} and the following fact:
\begin{equation} \label{le2-11-2}
\sum_{\emptyset \ne I\subseteq S}
\left (
\left (\sum_{i\in S\setminus I}d_i \right )
\left (\prod_{j\in I}a_j\right )
\left(\prod_{r\in S\setminus I}b_r\right )
\right )
=\sum_{\emptyset \ne I\subseteq S}
\left (
\left (\sum_{i\in I}d_i \right )
\prod_{j\in S\setminus I}a_j
\prod_{r\in I}b_r
\right )-
\left( \prod_{r\in S}b_r\right )\cdot \left( \sum_{i\in S}d_i\right ).
\end{equation}
\end{proof}

\begin{lemma}\label{le2-2}
Let $S$ be a set of positive integers.
For any set of $4|S|$ real numbers
$\{a_i,b_i,c_i,d_i: i\in S\}$,
\begin{equation} \label{le2-2-1}
\sum_{\emptyset \ne I\subsetneq S}
\left (
\left (\sum_{i\in I}c_i \right )
\left (\sum_{q\in S\setminus I}d_q \right )
\left (\prod_{j\in I}a_j\right )
\left(\prod_{r\in S\setminus I}b_r\right )
\right )
=\left(\prod_{r\in S}(a_r+b_r)\right)
\cdot
\sum_{i,j\in S\atop i\ne j}
\frac{c_ia_id_jb_j}{(a_i+b_i)(a_j+b_j)}.
\end{equation}
\end{lemma}

\begin{proof}
The result follows from the following fact:
\begin{equation} \label{le2-2-2}
\sum_{\emptyset \ne I\subsetneq S}
\left (
\left (\sum_{i\in I}c_i \right )
\left (\sum_{q\in S\setminus I}d_q \right )
\left (\prod_{j\in I}a_j\right )
\left(\prod_{r\in S\setminus I}b_r\right )
\right )
=\sum_{i,j\in S\atop i\ne j}
\left (
c_ia_id_jb_j\prod_{r\in S\setminus \{i,j\}}(a_r+b_r)
\right ).
\end{equation}
\end{proof}

By Lemmas~\ref{le2-11} and~\ref{le2-2}, for an arbitrary real number $c$, we have
\aln{le2-2-1ex}
{
&\sum_{\emptyset \ne I\subseteq S}
\left (
\left (c+\sum_{i\in I}c_i \right )
\left (\sum_{q\in S\setminus I}d_q \right )
\left (\prod_{j\in I}a_j\right )
\left(\prod_{r\in S\setminus I}b_r\right )
\right )
\nonumber \\
=&\left(\prod_{r\in S}(a_r+b_r)\right)
\cdot \left (
\sum_{i,j\in S\atop i\ne j}
\frac{c_ia_id_jb_j}{(a_i+b_i)(a_j+b_j)}
+
c\sum_{i\in S}
\frac{d_ib_i}{a_i+b_i}
\right )
-c\left( \prod_{r\in S}b_r\right )\cdot \left( \sum_{i\in S}d_i\right ).
}

\section{An identity}\label{sec3}

Define a function $\phi$ on $2k$ variables
$x_1,x_2,\ldots,x_k$ and $y_1,y_2,\ldots,y_k$, where $k\ge 1$,
as follows:
\begin{equation} \label{e3-3}
\phi(x_1,y_1,x_2,y_2,\ldots,x_k,y_k)
=\frac{1}{XY}\left(\prod_{i=1}^k(x_iY+y_iX)\right)
\left (1 -\sum_{i=1}^k\frac{x_iy_i}{x_iY+y_iX}\right ),
\end{equation}
where $X=x_1+x_2+\cdots+x_k$ and $Y=y_1+y_2+\cdots+y_k$.
Observe that \label{r2-6}
$$\phi(x_1,y_1,x_2,y_2,\ldots,x_k,y_k)=\frac{1}{XY} \left(\left(\prod\limits_{i=1}^k(x_iY+y_iX)\right)
-\sum\limits_{i=1}^k \left(x_iy_i\prod\limits_{1\le j\le k\atop j\ne i}(x_jY+y_jX)\right)\right).$$
In the expansion of
$\left(\prod\limits_{i=1}^k(x_iY+y_iX)\right)-\sum\limits_{i=1}^k \left(x_iy_i\prod\limits_{1\le j\le k\atop j\ne i}(x_jY+y_jX)\right)$,
the expression consisting of all monomials not divisible by $XY$ is identically $0$, as
shown below:
\eqn{XY-non}
{
& &Y^k\prod_{i=1}^k x_i
+X^k\prod_{i=1}^k y_i
-\sum_{i=1}^k\left ( y_i
Y^{k-1}\prod_{j=1}^k x_j
\right )
-\sum_{i=1}^k \left (x_i
X^{k-1}\prod_{j=1}^k y_j
\right )\nonumber  \\
&=&Y^k\prod_{j=1}^k x_j
+X^k\prod_{j=1}^k y_j
-Y^{k-1}\prod_{j=1}^k x_j
\cdot
\left (\sum_{i=1}^ky_i \right )
-X^{k-1}\prod_{j=1}^k y_j
\cdot
\left(\sum_{i=1}^k x_i\right) \nonumber  \\
&=&0.
}
It follows that $\phi(x_1, y_1, x_2, y_2, \ldots, x_k, y_k)$ is a polynomial on $2k$ variables
$x_1, x_2, \ldots, x_k$ and $y_1, y_2,\ldots,y_k$.

For $k\le 3$,
\equ{phi}
{
\left \{
\begin{array}{l}
\phi(x_1,y_1)=1;\\
\phi(x_1,y_1,x_2,y_2)=z_{1,2};\\ 
\phi(x_1,y_1,x_2,y_2,x_3,y_3)=
z_{1,2}z_{1,3}+z_{1,2}z_{2,3}+z_{1,3}z_{2,3},
\end{array}
\right.
}
where $z_{i,j}=x_iy_j+x_jy_i$
for all $1\le i<j\le 3$.
For any $I\subseteq \brk{k}\setminus \{1\}$,
let
\begin{equation} \label{e3-2}
x_I=x_1+\sum_{i\in I}x_i,\quad
y_I=y_1+\sum_{i\in I}y_i.
\end{equation}
In this section, we shall establish the following identity,
which will be applied to prove the main result in the article.

\begin{theorem}\label{th3-1}
For any $2k$ real numbers $x_1,x_2,\ldots,x_k$
and $y_1,y_2,\ldots,y_k$, where $k\ge 2$,
\begin{equation} \label{e3-4}
\phi(x_1,y_1,x_2,y_2,\ldots,x_k,y_k)
=
\sum_{\emptyset\ne I\subseteq \brk{k}\setminus \{1\}}
\left (
(-1)^{|I|-1}
\left(\prod_{j\in I}(x_1y_j+x_jy_1)\right)
\phi(x_I, y_I, \underbrace{x_s,y_s}
_{\forall s\in \brk{k}\setminus (I\cup \{1\})})
\right ),
\end{equation}
where
$
\phi(x_I, y_I, \underbrace{x_s,y_s}
_{\forall s\in \brk{k}\setminus (I\cup \{1\})} )
=\phi(x_I,y_I,x_{i_1},y_{i_2},\ldots,x_{i_r},y_{i_r}),
$
$\{i_1,i_2,\ldots,i_r\}=\brk{k}\setminus (I\cup \{1\})$ and $r=k-1-|I|$.
\end{theorem}

\begin{proof}
Let $I$ be a non-empty subset of $\brk{k}\setminus \{1\}$. Then
\equ{th3-1-1}
{x_I+\sum_{s\in \brk{k}\setminus (I\cup \{1\})}x_s=\sum_{i=1}^kx_i=X
}
and
\equ
{th3-1-2}
{
y_I+\sum_{s\in \brk{k}\setminus (I\cup \{1\})}y_s=\sum_{i=1}^ky_i=Y.
}
In the remainder of the proof of Theorem~\ref{th3-1},
let $W_I=x_IY+y_IX$, and for each
$i\in \brk{k}$, let
$$
W_i=x_iY+y_iX\ \mbox{and}\
w_i=x_iy_1+y_ix_1.
$$
By the definition of the function $\phi$,
\aln{e3-5}
{
\phi(x_I, y_I, \underbrace{x_s,y_s}_
{\forall s\in \brk{k}\setminus (I\cup \{1\})} )
&=\frac{W_I} 
{XY}\cdot
\left(\prod_{i\in \brk{k}\setminus (I\cup \{1\})}W_i
\right)
\left (1 -\frac{x_Iy_I}{W_I}
-\sum_{i\in \brk{k}\setminus (I\cup \{1\})}
\frac{x_iy_i}{W_i} 
\right ).
}
Thus, the right-hand side of (\ref{e3-4})
can be expressed as
\equ{eq3-50}
{
\frac {1}{XY}\left (\Gamma_1-\Gamma_2-\Gamma_3 \right ),
}
where
\begin{eqnarray}
	\Gamma_1&=&\sum_{\emptyset\ne I\subseteq \brk{k}\setminus \{1\}}
	\left ((-1)^{|I|-1}
	W_I 
	\left (\prod_{j\in I}
	w_j
	\right )
	\left (\prod_{i\in \brk{k}\setminus (I\cup \{1\})} W_i
	\right )
	\right ), \label{eq3-51} \\
\Gamma_2&=&\sum_{\emptyset\ne I\subseteq \brk{k}\setminus \{1\}}
\left ( (-1)^{|I|-1}
x_Iy_I
\left (\prod_{j\in I}
w_j 
\right )
\left (\prod_{i\in \brk{k}\setminus (I\cup \{1\})} W_i
\right )
\right ), \label{eq3-52} \\
\Gamma_3&=&\sum\limits_{\emptyset\ne I\subseteq \brk{k}\setminus \{1\}}
\left ( (-1)^{|I|-1}
W_I 
\left(\prod\limits_{j\in I}
w_j
\right )
\left(\prod\limits_{i\in \brk{k}\setminus (I\cup \{1\})} W_i 
\right )
\left ( \sum\limits_{i\in \brk{k}\setminus (I\cup \{1\})}\frac{x_iy_i}
{W_i} 
\right )
\right ).\label{eq3-53}
\end{eqnarray}
In the following, we shall apply Lemmas~\ref{le2-01}--~\ref{le2-2} to simplify $\Gamma_1, \Gamma_2$ and $\Gamma_3$
in order to show that
\equ{gamma-sum}
{\Gamma_1-\Gamma_2-\Gamma_3=
\left(\prod_{i=1}^k(x_iY+y_iX)\right)
\left (1-\sum_{i=1}^k\frac{x_iy_i}
{x_iY+y_iX}\right).
}
Let
$X'=X-x_1$ and $Y'=Y-y_1$.
Also let
\equ{3-terms}
{
\left \{
\begin{array}{l}
Z=\prod\limits_{i=2}^k(x_iY'+y_iX'),\\
Z_1=\sum\limits_{i=2}^k\frac{x_iy_i}{x_iY'+y_iX'},
\quad
Z_2=\sum\limits_{i=2}^k\frac{x_i^2}{x_iY'+y_iX'},
\quad
Z_3=\sum\limits_{i=2}^k\frac{y_i^2}{x_iY'+y_iX'}.
\end{array}
\right.
}
Note that for any non-empty subset $I$ of $\brk{k}\setminus \{1\}$,
\equ{term12}
{
y_IX+x_IY-x_Iy_I=\left(y_1+\sum_{i\in I}y_i\right)X
+
\left (x_1+\sum_{i\in I}x_i\right )
\left (\sum_{i\in \brk{k}\setminus (I\cup \{1\})}y_i\right ).
}
In the remainder of this section, let $W'_i$ denote $x_iY'+y_iX'$ for each
$i\in \brk{k}$.
By applying identities \label{r2-8}
(\ref{le2-1-1ex}), (\ref{le2-2-1ex})
and (\ref{term12}),
\begin{eqnarray}\label{eq3-61}
	& &\Gamma_1-\Gamma_2
	\nonumber \\
	&=&\sum_{\emptyset\ne I\subseteq \brk{k}\setminus \{1\}}
	\left ((-1)^{|I|-1}
	( W_I  
	-x_Iy_I)
	\left (\prod_{j\in I}
	w_j 
	\right )
	\left (\prod_{i\in \brk{k}\setminus (I\cup \{1\})} W_i 
	\right )
	\right ) \nonumber\\
	&\overset{by\ (\ref{term12})}{=}&
	X\sum_{\emptyset\ne I\subseteq \brk{k}\setminus \{1\}}
	\left ( (-1)^{|I|-1}
	\left (y_1+\sum_{i\in I}y_i\right )
	\left (\prod_{j\in I}
	w_j 
	\right )
	\left (
	\prod_{i\in \brk{k}\setminus (I\cup \{1\})}
	W_i  
	\right )
	\right )
	\nonumber \\
	& &+\sum_{\emptyset\ne I\subseteq \brk{k}\setminus \{1\}}
	\left ( (-1)^{|I|-1}
	\left (x_1+\sum_{i\in I}x_i\right )
	\left (\sum_{i\in \brk{k}\setminus (I\cup \{1\})}y_i\right )
	\left (\prod_{j\in I}
	w_j 
	\right )
	\left (
	\prod_{i\in \brk{k}\setminus (I\cup \{1\})}
	W_i 
	\right )
	\right )
	\nonumber \\
	&=&-X\sum_{\emptyset\ne I\subseteq \brk{k}\setminus \{1\}}
	\left (
	\left (y_1+\sum_{i\in I}y_i\right )
	\left (\prod_{j\in I}
	(-w_j) 
	\right )
	\left (
	\prod_{i\in \brk{k}\setminus (I\cup \{1\})}
	W_i 
	\right )
	\right )
	\nonumber \\
	& &-\sum_{\emptyset\ne I\subseteq \brk{k}\setminus \{1\}}
	\left (
	\left (x_1+\sum_{i\in I}x_i\right )
	\left (\sum_{i\in \brk{k}\setminus (I\cup \{1\})}y_i\right )
	\left (\prod_{j\in I}
	(-w_j) 
	\right )
	\left (
	\prod_{i\in \brk{k}\setminus (I\cup \{1\})}
	W_i 
	\right )
	\right )
	\nonumber \\
	&\overset{\small by\ (\ref{le2-1-1ex}), (\ref{le2-2-1ex})}
	{=}&
	Xy_1\prod_{i=2}^kW_i 
	-X\prod_{j=2}^k W'_j 
	\left (
	y_1-\sum_{i=2}^k
	\frac{y_i w_i 
	}
	{W'_i} 
	\right )
	+x_1\left(\prod_{i=2}^kW_i\right) 
	\left ( \sum_{j=2}^k y_j\right )
	\nonumber \\
	& &+\left ( \prod_{r=2}^k W'_r\right) 
	\left (
	\sum_{2\le i,j\le k\atop i\ne j}
	\frac{x_iy_j
		w_i W_j 
	}{W'_iW'_j}
	-x_1\sum_{i=2}^k
	\frac{y_i W_i}{W'_i}
	\right )
	\nonumber \\
	&=&\prod_{i=1}^kW_i 
	-x_1y_1\prod_{i=2}^k W_i 
	-XZ
	\left (y_1-x_1Z_3-y_1Z_1\right )
	+Z\left (
		-x_1(YZ_1+XZ_3)+
	\sum_{2\le i,j\le k\atop i\ne j}
	\frac{x_iy_jw_i W_j
	}
	{W'_i W'_j}
	\right )
	\nonumber \\
	&=&\prod_{i=1}^k W_i 
	-x_1y_1\prod_{i=2}^k
	W_i 
	-Z\left (y_1X-y_1XZ_1+x_1YZ_1\right )
	+Z\sum_{2\le i,j\le k\atop i\ne j}
	\frac{x_iy_j
		w_i W_j 
	}
	{W'_i W'_j},
\end{eqnarray}
where in the second last equality, we combine $Xy_1\prod\limits_{i=2}^k W_i$
and $x_1\prod\limits_{i=2}^k W_i \left (\sum\limits_{j=2}^k y_j\right )$ to obtain
$\prod\limits_{i=1}^kW_i-x_1y_1\prod\limits_{i=2}^k W_i$.
\label{r2-7}

By applying identity (\ref{le2-2-1ex}),
we have
\begin{eqnarray}\label{eq3-63}
\Gamma_3&=&-\sum\limits_{\emptyset\ne I\subseteq \brk{k}\setminus \{1\}}
\left (
\left( W_1 
+\sum\limits_{i\in I} W_i 
\right)
\left(\sum\limits_{i\in \brk{k}\setminus (I\cup \{1\})}
\frac{x_iy_i}{W_i} 
\right)
\left (\prod\limits_{j\in I}
(-w_j) 
\right )
\left(\prod\limits_{i\in \brk{k}\setminus (I\cup \{1\})} W_i 
\right)
\right )
\nonumber \\
&\overset{by\ (\ref{le2-2-1ex})}{=}&
W_1 
\left(\prod\limits_{i=2}^k
W_i 
\right)
\left (\sum_{i=2}^k\frac{x_iy_i}
{W_i}\right ) 
+\left(\prod_{i=2}^k
W'_i 
\right)
\left ( -W_1 
\sum_{i=2}^k\frac{x_iy_i}
{W'_i}  
+\left. \sum_{2\le i,j\le k\atop i\ne j}
\frac{
	w_i W_i x_jy_j
}
{W'_i W'_j
}
\right )
\right.
\nonumber \\
&=&\left(\prod\limits_{i=1}^k
W_i 
\right)
\left (\sum_{i=2}^k
\frac{x_iy_i}{W_i}\right ) 
-W_1 
ZZ_1
+ Z\sum_{2\le i,j\le k\atop i\ne j}
\frac{w_iW_ix_jy_j}
{W'_i W'_j 
}.
\end{eqnarray}

Note that for any set $S$ of positive integers,
and real numbers $a_i,b_i$ for $i\in S$, we have
$$
\sum_{i,j\in S\atop i\ne j}(a_ib_j)
=\left (\sum_{i\in S}a_i\right )\cdot
\left( \sum_{i\in S}b_i \right )
-\sum_{i\in S}(a_ib_i).
$$
Thus,
\begin{eqnarray}\label{eq3-7-1}
& &\sum_{2\le i,j\le k\atop i\ne j}
\frac{x_iy_j w_iW_j}  
{W'_i W'_j} 
-\sum_{2\le i,j\le k\atop i\ne j}
\frac{w_iW_ix_jy_j}   
{W'_iW'_j} 
\nonumber \\
&=&\left( \sum_{i=2}^k
\frac{x_iw_i}  
{W'_i}\right )    
\cdot
\left(\sum_{i=2}^k \frac{y_iW_i}
{W'_i}\right ) 
-\sum_{i=2}^k \frac{x_iy_iw_iW_i}
{(W'_i)^2}
-\left(
\left (\sum_{i=2}^k
\frac{x_iy_i}{W'_i} \right )  
\cdot
\left (\sum_{i=2}^k
\frac{w_iW_i}{W'_i} \right ) 
-\sum_{i=2}^k
\frac{x_iy_iw_iW_i} 
{(W'_i)^2}
\right)
\nonumber \\
&=&(x_1Z_1+y_1Z_2)(Y'+y_1Z_1+x_1Z_3)
-Z_1(x_1Y'+y_1X'+x_1^2Z_3+y_1^2Z_2+2x_1y_1Z_1)
\nonumber \\
&=&y_1(Y'Z_2-X'Z_1)
+x_1y_1(Z_2Z_3-Z_1^2)
\nonumber \\
&=&y_1(Y'Z_2-X'Z_1)
+x_1y_1(1-2Z_1)
\nonumber \\
&=&y_1(Y'Z_2+x_1-X'Z_1-2x_1Z_1),
\end{eqnarray}
where the second last equality follows from
the fact that $(Z_1-1)^2=Z_2Z_3$ by Lemma~\ref{le2-01}.
Thus, by (\ref{eq3-61}), (\ref{eq3-63})
and (\ref{eq3-7-1}),
\begin{eqnarray}\label{eq3-12}
& &\Gamma_1-\Gamma_2-\Gamma_3
-\left(\prod_{i=1}^k(x_iY+y_iX)\right)
\left (1-\sum_{i=1}^k\frac{x_iy_i}{x_iY+y_iX}\right )
\nonumber \\
&=&-Z\left (y_1X-y_1XZ_1+x_1YZ_1\right )
+ZZ_1(x_1Y+y_1X)+Zy_1(Y'Z_2+x_1-X'Z_1-2x_1Z_1)
\nonumber \\
&=&y_1Z(X'Z_1+Y'Z_2-X')
\nonumber \\
&=&y_1Z
\left (-X'+\sum_{i=2}^k
\frac{x_i(y_iX'+x_iY')}{x_iY'+y_iX'}\right )
\nonumber \\
&=&0.
\end{eqnarray}
Thus, (\ref{e3-4}) follows from (\ref{eq3-50}), (\ref{eq3-12})
and the definition of $\phi(x_1,y_1,x_2,y_2,\ldots,x_k,y_k)$
in (\ref{e3-3}).
\end{proof}

\section{Counting spanning trees in a
special type of multigraphs}
\label{sec4}

Let $V=\{v_i: i\in \brk{k}\}$  and
$E=\{v_iv_j: i,j\in \brk{k}~\text{and}~i\neq j\}$ be the vertex set and edge set of
the complete graph $K_k$, where $k\ge 1$.
Let $\omega$ be a weight function on $E$.
If $\omega(v_iv_j)$ is a nonnegative integer
for all $i,j$ with $1\le i<j\le k$,
then $\tau(K_k, \omega)$ is the number of spanning trees
of the multigraph with vertex set
$\{u_1,u_2,\ldots,u_k\}$ which contains
exactly $\omega(v_iv_j)$ parallel edges
joining $u_i$ and $u_j$
for all $i,j$ with $1\le i<j\le k$.

For any non-empty subset
$I$ of $\brk{k}\setminus \{1\}$,
let $G_I$ denote the complete graph of order $k-|I|$
with vertex set $\{v_I\}\cup \{v_i: i\in \brk{k}\setminus (I\cup \{1\})\}$
and weight function $\omega_I$ on the edge set
of $G_I$ defined as follows:
\equ{weight-f}
{
\left \{
\begin{array}{ll}
\omega_I(v_iv_j)=\omega(v_iv_j),
&\forall i,j\in \brk{k}\setminus (I\cup \{1\}),i\ne j;\\
\omega_I(v_Iv_j)=\sum\limits_{r\in I\cup \{1\}}\omega(v_rv_j),
\qquad
&\forall j\in \brk{k}\setminus (I\cup \{1\}).
\end{array}
\right.
}
Note that $G_I$ is actually the graph obtained from $K_k$ by identifying
all vertices in $\{v_i: i\in
I\cup \{1\}\}$ as one vertex.

By the inclusion-exclusion principle,
the following recursive relation on $\tau(K_k,\omega)$
can be obtained.\label{r2-9}

\begin{lemma}\label{le4-1}
For any weight function $\omega$ on the edge set $E$ of $K_k$, \label{r2-10}
\begin{equation}\label{le4-1-e1}
\tau(K_k,\omega)
=\sum_{\emptyset \ne I\subseteq \brk{k}\setminus \{1\}}
\left ( (-1)^{|I|-1}
\tau(G_I, \omega_I)
\prod_{i\in I}\omega(v_1v_i)\right ).
\end{equation}
\end{lemma}

\def \C {{\cal C}}
\def \F {{\cal F}}

\begin{proof} For any $i\in \brk{k}\setminus \{1\}$, let $A_i$
denote the set of members $T$ in $\mathcal{T}(K_k)$ with
$v_1v_i\in E(T)$.
Clearly,
\equ{le4-1-e2}
{
\mathcal{T}(K_k)=\bigcup_{i=2}^k A_i.
}
For any $T\in \mathcal{T}(K_k)$, by the inclusion-exclusion principle,
\equ{le4-1-e3}
{
\sum_{\emptyset\ne I\subseteq \brk{k}\setminus \{1\}}
\left ((-1)^{|I|-1}|\{T\}\cap \bigcap_{i\in I} A_i|
\right )=1.
}
Thus,
\equ{le4-1-e4}
{
\tau(K_k,\omega)=\sum_{T\in \bigcup_{2\le i\le k}A_i}
\omega(T)
=\sum_{\emptyset \ne I\subseteq \brk{k}\setminus \{1\}}
\left ( (-1)^{|I|-1}
\sum_{T\in \bigcap_{i\in I} A_i}
\omega(T)\right ).
}
Let $I$ be  any non-empty subset of $\brk{k}\setminus \{1\}$ and
let $H_I$ denote the multiple graph
obtained from $K_k$ by
identifying all vertices
in the set $\{v_i: i\in I\cup \{1\}\}$
and removing all loops produced.
The vertex set of $H_I$ is
$(V(K_k)\setminus \{v_i:i\in I\cup \{1\}\})\cup v_I$.
Clearly, $H_I$ includes each edge
$v_iv_j$, where
$i,j\in \brk{k}\setminus (I\cup \{1\})$,
while each edge $v_iv_j$ in $K_k$,
where $i\in I\cup \{1\}$ and
$j\in \brk{k}\setminus (I\cup \{1\})$,
is changed to an edge of $H_I$
joining $v_I$ and $v_j$.
There are exactly $1+|I|$
parallel edges
in $H_I$ joining $v_I$ and $v_j$
for each $j\in \brk{k}\setminus (I\cup \{1\})$.
The weight function on $E(H_I)$
is the restriction of $\omega$ to $E(H_I)$ and parallel edges in
$H_I$ may have different weights.

For each $T\in \bigcap_{i\in I} A_i$,
let $T_I$ be the tree obtained from $T$ by identifying all vertices
in the set $\{v_i: i\in I\cup \{1\}\}$.
Clearly $\omega(T)$ and $\omega(T_I)$
have the following relation:
\equ{le4-1-e7}
{
\omega(T)=\prod_{i\in I}\omega(v_1v_i)
\cdot
\omega(T_I).
}
Moreover,
$T\rightarrow T_I$ is a bijection
 from
$\bigcap_{i\in I} A_i$ to
$\mathcal{T}(H_I)$, implying that
\equ{le4-1-e5}
{
\sum_{T\in \bigcap_{i\in I} A_i}\omega(T)
=\prod_{i\in I}\omega(v_1v_i)
\cdot \sum_{T\in \bigcap_{i\in I} A_i}\omega(T_I)
=\tau(H_I, \omega)\cdot
\prod_{i\in I}\omega(v_1v_i).
}
Note that $G_T$ can be obtained from $H_I$ by merging all parallel edges
with ends $v_I$ and $v_j$ into one
for each
$j\in \brk{k}\setminus (I\cup \{1\})$.
By the definition of $\omega_I$,
$\tau(H_I, \omega)
=\tau(G_I, \omega_I)$.
Thus, (\ref{le4-1-e1}) follows from (\ref{le4-1-e4}) and (\ref{le4-1-e5}).
\label{r1-2}
\end{proof}

Recall the function $\phi$ defined in the previous
section. In the following, we shall show that
$\tau(K_k, \omega)$ can be expressed in terms of
$\phi$ when $\omega$ satisfies certain conditions.\label{r2-13}

\begin{theorem}\label{th3}
Let $V=\{v_1,v_2,\ldots,v_k\}$ be the
the vertex set of the complete graph $K_k$, where $k\ge 1$,
and $\omega$ be a weight function on the edge set $E$ of $K_k$.
If there exist $2k$ real numbers $x_1,x_2,\ldots,x_k$
and $y_1,y_2,\ldots,y_k$ such that
$\omega(v_iv_j)=x_iy_j+x_jy_i$ holds
for every pair $i$ and $j$ with $1\le i<j\le k$,
then,
\begin{equation}\label{th3-e1}
\tau(K_k, \omega)
=\phi(x_1,y_1,x_2,y_2,\ldots,x_k,y_k).
\end{equation}
\end{theorem}

\begin{proof} Note that for
 $1\le k\le 3$,
\equ{eq4-0}
{
\tau(K_k, \omega)=
\left \{
\begin{array}{ll}
1,  &k=1;\\
\omega(v_1v_2), &k=2;\\
\omega(v_1v_2)\omega(v_1v_3)
+\omega(v_1v_2)\omega(v_2v_3)
+\omega(v_1v_3)\omega(v_2v_3), &k=3.
\end{array}
\right.
}
As $\omega(v_iv_j)=x_iy_j+x_jy_i$,
(\ref{th3-e1}) follows from (\ref{phi}) and
(\ref{eq4-0}) when $k\le 3$.

Assume that the result holds for $k\le N$, where $N\ge 3$.
In the following, we assume that $k=N+1$
and show that it holds in this case by induction.\label{r2-15}

Recall that for any non-empty subset of $I$ of $\brk{k}\setminus \{1\}$,
$G_I$ is the complete graph
with vertex set $\{v_I\}\cup \{v_j: j\in \brk{k}\setminus (I\cup \{1\})\}$ and
weight function $\omega_I$ on its edge set
defined in (\ref{weight-f}), i.e.,
\equ{weight-f2}
{
\left \{
\begin{array}{ll}
\omega_I(v_iv_j)=\omega(v_iv_j)=x_iy_j+x_jy_i,
&\forall i,j\in \brk{k}\setminus (I\cup \{1\}),i\ne j;\\
\omega_I(v_Iv_j)
=\sum\limits_{r\in I\cup \{1\}}(x_ry_j+x_jy_r)
=x_Iy_j+y_Ix_j,
\quad
&\forall j\in \brk{k}\setminus (I\cup \{1\}).
\end{array}
\right.
}
where $x_I=x_1+\sum\limits_{r\in I}x_r$
and $y_I=y_1+\sum\limits_{r\in I}y_r$.

As $G_I$ is a complete graph of order $k-|I|<k$
with a weight function $\omega_I$ satisfying conditions
in (\ref{weight-f2}),
by inductive assumption,
(\ref{th3-e1}) holds for $G_I$, i.e.,
\begin{equation}\label{th3-e5}
\tau(G_I, \omega_I)
=\phi(x_I,y_I, \underbrace{x_i,y_i}
_{\forall i\in \brk{k}\setminus (I\cup \{1\})}).
\end{equation}

By Lemma~\ref{le4-1}, (\ref{th3-e5}) and Theorem~\ref{th3-1},
\eqn{th3-e6}
{
\tau(K_k, \omega)
&=&\sum_{\emptyset\ne I\subseteq \brk{k}\setminus \{1\}} (-1)^{|I|-1}
\tau(G_I, \omega_I)
\prod_{i\in I}\omega(v_1v_i)
\nonumber \\
&=&\sum_{\emptyset \ne I\subseteq \brk{k}\setminus \{1\}}
\left ( (-1)^{|I|-1}
\tau(G_I, \omega_I)
\prod_{i\in I}(x_iy_1+y_ix_1)
\right )
\nonumber \\
&=&\sum_{\emptyset \ne I\subseteq \brk{k}\setminus \{1\}}
\left ( (-1)^{|I|-1}
\prod_{i\in I}(x_iy_1+y_ix_1)
\cdot
\phi(x_I,y_I, \underbrace{x_i,y_i}
_{\forall i\in \brk{k}\setminus (I\cup \{1\})})
\right )
\nonumber \\
&=&\phi(x_1,y_1,x_2,y_2,\cdots,x_k,y_k).
}
Hence the result holds.
\end{proof}

\section{Spanning trees in $K_{m, n}$ containing a spanning forest
$F$}
\label{sec5}

Now we are ready to prove the main result.

\noindent
{\bf Proof of Theorem \ref{mainthm}.}
For any spanning forest $F$ of $K_{m,n}$
with components $T_1,T_2,\ldots,T_k$,
observe that
\equ{eq5-1}
{
\tau_{F}(K_{m,n})=\tau(K_{m,n}/F),
}
where $K_{m,n}$ is unweighted and $K_{m,n}/F$ is \label{r2-16}
the multigraph obtained from $K_{m,n}$
by contracting all edges in $F$.
Note that $K_{m,n}/F$ is a multigraph of order $k$
whose vertices correspond to components of $F$,
as $K_{m,n}/F$ can also be obtained from $K_{m,n}$
by identifying all vertices in $T_i$
for all $i\in \brk{k}$, and removing all loops.
Thus, we may assume that $K_{m,n}/F$ has vertices
$v_1,v_2,\ldots,v_k$ such that the number of parallel
edges joining $v_i$ and $v_j$ is equal to
the number of edges in $K_{m,n}$
with one end in $T_i$ and the other end in $T_j$.

As $|X\cap V(T_s)|=m_s$ and $|Y\cap V(T_s)|=n_s$
for all $s=1,2,\ldots,k$,
$K_{m,n}/F$ contains exactly $m_in_j+m_jn_i$ parallel
edges joining $v_i$ and $v_j$
for all $1\le i<j\le k$.
By Theorem~\ref{th3},
\equ{eq5-4}
{
\tau(K_{m,n}/F)=
\phi(m_1,n_1,m_2,n_2,\ldots,m_k,n_k).
}
Thus, by (\ref{eq5-1}) and the definition of $\phi(x_1,y_1,x_2,y_2,\ldots,x_k,y_k)$
in (\ref{e3-3}), the result holds.
{\hfill$\Box$}

\section{Remarks}
\label{sec6}

Another approach for proving the main result is to
establish results analogue to Lemma~\ref{le4-1} and
Theorem~\ref{th3-1}.
The following identity analogue to Lemma~\ref{le4-1}
can be obtained easily: 
\begin{equation}\label{e6-1}
\tau(K_k, \omega)
=\sum_{\emptyset \ne I\subseteq \brk{k}\setminus \{1\}}
\left (
\tau(G_I, \omega'_I)
\prod_{i\in I}\omega(v_1v_i)\right ),
\end{equation}
where $\omega'_I$ is different from
$\omega_I$ defined in (\ref{weight-f}), as
for any $j\in \brk{k}\setminus (I\cup \{1\})$,
\equ{e6-2}
{
\omega'_I(v_Iv_j)=\sum_{r\in I}\omega(v_rv_j)
=\omega_I(v_Iv_j)-\omega(v_1v_j),
}
although $\omega'_I(v_iv_j)=\omega_I(v_iv_j)$
for all $i,j\in \brk{k}\setminus (I\cup \{1\})$ with $i\ne j$.

By Theorem~\ref{mainthm} and (\ref{e6-1}),
the following identity analogue to
Theorem~\ref{th3-1} holds:
\begin{equation} \label{e6-3}
\phi(x_1,y_1,x_2,y_2,\cdots,x_k,y_k)
=
\sum_{\emptyset\ne I\subseteq \brk{k}\setminus \{1\}}
\left (
\prod_{j\in I}(x_1y_j+x_jy_1)
\phi(x'_I, y'_I, \underbrace{x_s,y_s}
_{\forall s\in \brk{k}\setminus (I\cup \{1\})})
\right ),
\end{equation}
where $x'_I=\sum\limits_{i\in I}x_i=x_I-x_1$
and $y'_I=\sum\limits_{i\in I}y_i=y_I-y_1$.

However, it is quite challenging to prove (\ref{e6-3})
directly.
Note that for any $I$ with $\emptyset \ne I\subseteq \brk{k}\setminus \{1\}$,
\equ{e6-4}
{
\left \{
\begin{array}{l}
x'_I+\sum_{i\in I}x_i=x_2+x_3+\cdots+x_k=X-x_1,\\
y'_I+\sum_{i\in I}y_i=y_2+y_3+\cdots+y_k=Y-x_1.
\end{array}
\right.
}
By the definition of the function $\phi$,
\aln{e6-5}
{
\phi(x'_I, y'_I, \underbrace{x_s,y_s}
_{\forall s\in \brk{k}\setminus (I\cup \{1\})} )
&=\frac{x'_IY'+y'_IX'}{X'Y'}\cdot
\left(\prod_{i\in \brk{k}\setminus (I\cup \{1\})}(x_iY'+y_iX')\right)
\nonumber \\
& \cdot
\left (1 -\frac{x'_Iy'_I}{x'_IY'+y_I'X'}
-\sum_{i\in \brk{k}\setminus (I\cup \{1\})}\frac{x_iy_i}
{x_iY'+y_iX'}\right ),
}
where $X'=X-x_1$ and $Y'=Y-y_1$.
Observe that
the left-hand of (\ref{e6-3}) has a denominator $XY$,
while its right-hand side has a denominator $X'Y'$.

Clearly, the main result (i.e., Theorem~\ref{mainthm})
also follows from (\ref{e6-1}) and
(\ref{e6-3}).

In the end, we propose some problems.

\begin{problem}\label{prob1}
Find a bijective proof for Theorem~\ref{mainthm}.
\end{problem}

Another problem is to extend Theorem~\ref{mainthm} to complete
$k$-partite graphs, where $k\ge 3$.

\begin{problem}\label{prob2}
Let $K_{n_1,n_2,\cdots,n_k}$ be a complete $k$-partite graph
and $F$ be a spanning forest in $K_{n_1,n_2,\ldots,n_k}$,
where $k\ge 3$.
Find a formula for counting the number of spanning trees in
$K_{n_1,n_2,\ldots,n_k}$ which contain all edges in $F$.
\end{problem}

For $k=3$, we propose the following conjecture for
a lower bound of $\tau_F(K_{n_1,n_2,n_3})$.

\begin{conjecture}\label{con1}
Let $X_1,X_2$ and $X_3$ be the partite sets of
the complete tripartite graph
$K_{n_1,n_2,n_3}$,  where $|X_i|=n_i$ for $i\in \brk{3}$.
For any spanning forest $F$ in $K_{n_1,n_2,n_3}$
with $k$ components $T_1,T_2,\cdots,T_k$,
\eqn{con-e1}
{
\tau_F(K_{n_1,n_2,n_3})
&\ge &
\frac 1{n_1n_2+n_1n_3+n_2n_3}
\left (
\prod_{i=1}^{k}\left(
(n-n_1)n_{1,i}+(n-n_2)n_{2,i}+(n-n_3)n_{3,i} \right )
\right )
\nonumber \\
& &\cdot \left (
1-\sum_{i=1}^{k}
\frac {n_{1,i}n_{2,i}+n_{1,i}n_{3,i}+n_{2,i}n_{3,i}}
{(n-n_1)n_{1,i}+(n-n_2)n_{2,i}+(n-n_3)n_{3,i}}
\right ),
}
where $n=n_1+n_2+n_3$ and $n_{s,i}=|X_s\cap V(T_i)|$ for
$s=1,2,3$ and $i\in \brk{k}$.
\end{conjecture}

It is trivial to verify that
the equality of (\ref{con-e1}) holds for $k\le 2$.

\section*{Acknowledgements}

Jun Ge is supported by NSFC (No. 11701401) and
the joint research project of Laurent Mathematics Center of Sichuan Normal University
and National-Local Joint Engineering Laboratory of System Credibility Automatic Verification.
The authors are grateful to two anonymous referees for their careful examination and constructive comments.

\end{document}